\documentclass[12pt]{amsart}
\usepackage[]{hyperref}
\usepackage{amssymb,amsfonts,amsmath,amsopn,amstext,amscd,latexsym, amsthm, enumerate,mathrsfs}
\usepackage{lipsum}
\usepackage[margin=32mm]{geometry}
\usepackage{bbm}
\usepackage[all,cmtip]{xy}

\usepackage{enumerate}
\usepackage[shortlabels]{enumitem}

\newtheorem{theorem}{Theorem}[section]

\newtheorem*{thmA}{Theorem~A}
\newtheorem*{thmB}{Theorem~B}
\newtheorem*{thmC}{Theorem C}
\newtheorem*{thmD}{Theorem D}
\newtheorem*{thmE}{Theorem E}
\newtheorem*{thmF}{Theorem F}

\newtheorem{lem}[theorem]{Lemma}

\newtheorem{cor}[theorem]{Corollary}

\theoremstyle{remark}

\newcommand{\Aut}{{\mathrm {Aut}}}
\newcommand{\Out}{{\mathrm {Out}}}
\newcommand{\Centralizer}{{\mathbf {C}}}
\newcommand{\Center}{{\mathbf {Z}}}
\newcommand{\PSL}{{\mathrm {PSL}}}
\newcommand{\PSU}{{\mathrm {PSU}}}
\newcommand{\SL}{{\mathrm {SL}}}

\newcommand{\GL}{{\mathrm {GL}}}

\newcommand{\OO}{\mathbf{O}}
\newcommand{\Alt}{\mathbf{A}}
\newcommand{\Sym}{\mathbf{S}}
\newcommand{\Syl}{\rm{Syl}}
\newcommand{\Normalizer}{\mathbf{N}}
\newcommand{\Sol}{\rm{Sol}}
\newcommand{\Fit}{\mathbf{F}}

\newcommand{\Comp}{\mathbf{E}}

\theoremstyle{definition}

\begin{document}
\title[Conjugacy classes of $p$-elements]{Conjugacy classes of $p$-elements and normal $p$-complements}

\author[H. P. Tong-Viet]{Hung P. Tong-Viet}
\address{Department of Mathematical Sciences, Binghamton University, Binghamton, NY 13902-6000, USA}
\email{tongviet@math.binghamton.edu}

%\thanks{}


%\date{\today}

%\keywords{character degrees; prime  graphs; triangles;}

\begin{abstract}  
In this paper, we study the structure of finite groups with a large number of conjugacy classes of $p$-elements for some prime $p$. As consequences, we obtain some new criteria for the existence of normal $p$-complements in finite groups.

\end{abstract}

\maketitle

\section{Introduction}
Let $p$ be a prime. Let $G$ be a finite group and let $P$ be a Sylow $p$-subgroup of $G$. Denote by $k(G)$ and $k_p(G)$ the number of conjugacy classes of $G$ and the number of conjugacy classes of $p$-elements of $G$, respectively.   By Sylow's theorem, we can choose a complete set $\Gamma$ of representatives  for the conjugacy classes of $p$-elements of $G$ in such a way that $\Gamma\subseteq P.$ This yields that $k_p(G)\leq k(P)$. Also $k_p(G)\ge 2$ unless $G$ is a $p'$-group. Hence if $p$ divides $|G|$, then  $2\leq k_p(G)\leq k(P)\leq |P|.$
In \cite{KNST}, the authors study finite groups $G$ with $k_p(G)=2$. They show that the Sylow $p$-subgroup $P$ of such a group $G$ must be either elementary abelian or extra-special of order $p^3.$  In this paper, we will look at the case when $k_p(G)$ is large in comparison to $|P|$.

Recall that a finite group is said to be $p$-nilpotent if it has a normal $p$-complement. A classical result in group theory states that a finite group $G$ is $p$-nilpotent if and only if $P$ controls its own fusion in $G$. (See \cite[5.25]{Isaacs} and the definitions in Section \ref{sec2}). The latter condition is equivalent to $x^G\cap P=x^P$ for every $x\in P$ which is equivalent to the condition $k_p(G)=k(P)$. Thus  $G$ is $p$-nilpotent if and only if $k_p(G)=k(P)$.
If we assume that $k_p(G)=|P|$, then $k_p(G)=k(P)=|P|$; hence $G$ is $p$-nilpotent and has an abelian Sylow $p$-subgroup. So, we may ask whether $G$ is still $p$-nilpotent,  if $k_p(G)/|P|$ is close to $1$.

It turns out that the fraction $k_p(G)/|P|$ is related to the commuting probability $d(G)$ of a group $G$, which is defined to be the probability that two randomly chosen elements of $G$ commute. Gustafson \cite{Gus} shows that $d(G)=k(G)/|G|$. The invariant $d(G)$ is also called the commutativity degree of $G$.

Here is our first result for even prime.
 \begin{thmA}\label{thmA}

Let $G$ be a finite group and let $P$ be a Sylow $2$-subgroup of $G$. Then   $k_2(G)> |P|/2$ if and only if $G$ has a normal $2$-complement and $k(P)>|P|/2$.
\end{thmA}

Clearly,  any groups in Theorem A are solvable by applying Feit-Thompson theorem. Also, the Sylow $2$-subgroup $P$ in Theorem A  is nilpotent of class at most $2.$   (See Lemma \ref{lem: classification}). Theorem A  does not hold if we allow equality. For example, if $G=\Alt_4$ and $P\in\Syl_2(G)$, then $k_2(G)=2$ and $|P|=4$, so $k_2(G)=|P|/2$ but $G$ is not $2$-nilpotent. Also, we cannot replace $2$ by an odd prime. Indeed, if $G=\Alt_5$ and $P\in\Syl_3(G)$, then $k_3(G)=2$ and $|P|=3$; hence $k_3(G)=\frac{2}{3}|P|>\frac{1}{2}|P|$ but $G$ is not $3$-nilpotent.

In view of Lemma \ref{lem: quotients}, to investigate the structure of finite groups $G$ with $k_p(G)/|P|$ being a constant, we may assume that $\OO_{p'}(G)=1$.

 \begin{thmB} 
Let $G$ be a finite  group and let $P\in\Syl_2(G)$. Suppose that $\OO_{2'}(G)=1$ and $k_2(G)=|P|/2$. Then
\begin{enumerate}[$(1)$]
\item $G/\Center(G)\cong\Alt_4$ or $\Sym_4$; or
\item  $G/\Center(G)$ is an almost simple group with a non-abelian simple socle isomorphic to $\PSL_2(q)$ with $3<q\equiv 3,5\, (\text{mod}\: 8)$.
\end{enumerate}
\end{thmB}

Let $\pi$ be a set of primes. Let $k_\pi(G)$ be the number of conjugacy classes of $\pi$-elements of $G$. Let $|G|_\pi$ be the $\pi$-part of the order $|G|$ of $G$. Define $d_\pi(G)$ to be $k_\pi(G)/|G|_\pi$. If $\pi=\{p\}$, then we write $d_p(G)$  and $k_p(G)$ instead of $d_{\{p\}}(G)$ or $k_{\{p\}}(G)$.  
We now  investigate the structure of finite groups $G$ with $d_\pi(G)>1/2$, where $\pi$ is a set of primes containing $2$. 

\begin{thmC} Let $G$ be a finite group and let $\pi$ be a set of primes with $2\in\pi$. Let $\sigma=\pi\setminus \{2\}$. Suppose that $d_\pi(G)>1/2$. Then $G$ has a normal $\pi$-complement and an abelian Hall $\sigma$-subgroup.
\end{thmC}

We should point out that our proofs of Theorems A$-$C do not depend on the classification of finite simple groups. 

For odd primes $p$, we obtain the following result, unfortunately, our proof depends on the odd version of Glauberman $Z^*$-theorem and thus depends on the classification of finite simple groups.
\begin{thmD} Let $G$ be a finite group and let $p$ be an odd prime. Then $d_p(G)>(p+1)/(2p)$ if and only if  $G$ has a normal $p$-complement and an abelian  Sylow $p$-subgroup.
\end{thmD}

This bound cannot be improved since $d_p(\textrm{D}_{2p})=(p+1)/(2p)$ but $\textrm{D}_{2p}$ is not $p$-nilpotent, where $p$ is an odd prime. For non-solvable examples, let $f\ge 2$ be an integer and $p$ be a prime such that $4^f-1$ is divisible by $p$ but not by $p^2$. Then $d_p(\PSL_2(2^f))=(p+1)/(2p).$

\begin{thmE} Let $G$ be a finite group and let $\pi$ be a set of odd primes. Let $p$ be the smallest prime in $\pi$. Suppose that $d_\pi(G)>(p+1)/(2p)$. Then $G$ has a normal $\pi$-complement and an abelian Hall $\pi$-subgroup.
\end{thmE}

In \cite{MH}, the authors show that if $d_\pi(G)>5/8$, then $d_\pi(G)=1$ or $2/3$. They also study the structure of finite groups $G$ such that $\OO_{3'}(G)=1$ and $d_3(G)=2/3$. Thus if $p=3$ in Theorem D, then our results follow immediately from their results. However, if $p\ge 5$, then  $(p+1)/(2p)<5/8$. Hence our Theorems C and E above improve their Theorem $1$ and finally our last result includes Theorem $2$ in \cite{MH}.

\begin{thmF} Let $G$ be a finite group and let $p$ be an odd prime. Let $P$ be a Sylow $p$-subgroup of $G$. Suppose that $\OO_{p'}(G)=1$ and $d_p(G)=(p+1)/(2p)$. Then $P$ is abelian, $\Normalizer_G(P)/\Centralizer_G(P)$ has order $2$, $[P,\Normalizer_G(P)]$ has order $p$ and $G\cong A\times B$, where $B$ is an abelian $p$-group and $A$ is either a dihedral group of order $2p$ or  an almost simple group with a Sylow $p$-subgroup of order $p$ contained in the socle of $A$.
\end{thmF}

The paper is organized as follows. We collection some results needed for the proofs of the main theorems in Section~\ref{sec2}. We prove Theorems A-C in Section~\ref{sec3} and prove Theorems E-F in Section~\ref{sec4}.
%--------------------------------------------
\section{Control of fusion and Glauberman $Z^*$-theorem}\label{sec2}
Let $G$ be a finite group and let $K\leq H\leq G$ be subgroups of $G$. We say that \emph{$H$ controls $G$-fusion in $K$} if and only if every pair of $G$-conjugate elements of $K$ are $H$-conjugate, that is, if $x,x^g\in K$ for some $g\in G,$ then $x^g=x^h$ for some $h\in H$. 
Let $p$ be a prime and let $H$ be a subgroup of $G$.  We say that $H$ \emph{controls $p$-fusion} in $G$ if $H$ contains a Sylow $p$-subgroup $P$ of $G$ and $H$ controls $G$-fusion in $P$. 
We first list some classical results on the existence of normal $p$-complements as well as the control of fusion in finite groups.
\begin{lem}\label{lem: fusion} Let $G$ be a finite group and let $P$ be a Sylow $p$-subgroup of $G$ for some prime $p$.
\begin{enumerate}[$(1)$]
\item $\Normalizer_G(P)$ controls $G$-fusion in $\Centralizer_G(P).$

\item  If $P\subseteq \Center(\Normalizer_G(P))$, then $G$ has a normal $p$-complement.
\item $G$ has a normal $p$-complement if and only if $P$ controls its own fusion in $G$.
\end{enumerate}
\end{lem}
\begin{proof} These are well-known results, for proofs, see Lemma 5.12, Theorems 5.13 and 5.25 in \cite{Isaacs}.
\end{proof}
Parts (1) and (2) above are known as Burnside's lemma and Burnside's normal $p$-complement theorem, respectively.
Here are some obvious consequences of the lemma.
\begin{cor}\label{cor}  Let $G$ be a finite group and let $P$ be a Sylow $p$-subgroup of $G$ for some prime $p$.
\begin{enumerate}[$(1)$]
\item $k_p(G)\leq k(P)$ and equality holds if and only if $G$ has a normal $p$-complement.
\item $k_p(G)=|P|$ or equivalently $d_p(G)=1$ if and only if $G$ has a normal $p$-complement and an abelian Sylow $p$-subgroup.
\end{enumerate}
\end{cor}
Note that part (1) of the corollary is equivalent to the statement that $d_p(G)\leq d(P)$ and equality holds if and only if $G$ has a normal $p$-complement.
The following result is a consequence of the definitions above and Sylow's theorem.

\begin{lem}\label{lem: controls of p-fusion}
Let $G$ be a finite group and let $P$ be a Sylow $p$-subgroup of $G$ for some prime $p$. Let $x\in P$. Then $x^G\cap P=\{x\}$ if and only if  $\Centralizer_G(x)$ controls $p$-fusion in $G$. 
\end{lem}

\begin{proof}
Let $x\in P$. Assume that $x^G\cap P=\{x\}$. We claim that $\Centralizer_G(x)$ controls $p$-fusion in $G$. Since $x^P\subseteq x^G\cap P=\{x\}$, we see that $x\in \Center(P)$ and thus $P\leq \Centralizer_G(x)$. Now  assume  that $y,y^g\in P$ for some $g\in G.$
We need to show that $y^g=y^h$ for some $h\in \Centralizer_G(x)$.
We have that $\{y,y^g\}\subseteq P\subseteq \Centralizer_G(x)$ which implies that $\{x,x^{g^{-1}}\}\subseteq\Centralizer_G(y)$. Let $U$ be a Sylow $p$-subgroup of $\Centralizer_G(y)$ containing $x$. By Sylow's theorem, $U\leq P^t$ for some $t\in G$. It follows that $x^{t^{-1}}\in P\cap x^G=\{x\}$; hence $x^{t^{-1}}=x$, so $t\in \Centralizer_G(x).$ Now $x^{g^{-1}}\in U^c$ for some $c\in \Centralizer_G(y)$ as $x^{g^{-1}}\in \Centralizer_G(y)$ is a $p$-element. We now have that $x^{g^{-1}c^{-1}t^{-1}}\in  P$ and thus $x^{g^{-1}c^{-1}t^{-1}}=x$ which implies that $g^{-1}c^{-1}\in \Centralizer_G(x)$. Therefore  $cg=h\in \Centralizer_G(x).$ Now $y^g=y^{cg}=y^h$ as wanted.

For the converse, let $P_1\in\Syl_p(G)$ and assume that $P_1\subseteq \Centralizer_G(x)$ and that $\Centralizer_G(x)$  controls $G$-fusion in $P_1$. It follows that  $x\in P_1$. By Sylow's theorem, $P=P_1^t$ for some $t\in G.$ Since $x\in P,$ $x^{t^{-1}}\in P_1\leq \Centralizer_G(x)$. As $\Centralizer_G(x)$ controls $G$-fusion in $P_1$, it follows that $x^{t^{-1}}=x^h$ for some $h\in\Centralizer_G(x)$. Hence  $x^{t^{-1}}=x^h=x$ and so $t\in\Centralizer_G(x)$. In particular, $P=P_1^t\subseteq \Centralizer_G(x)$. Finally, if  $x^g\in P$ for some $g\in G,$ then $x^g=x^h$ for some $h\in \Centralizer_G(x)$  and so $x^g=x^h=x.$ Therefore $x^G\cap P=\{x\}$.
\end{proof}

 For a finite group $G$ and a prime $p$, we define $\Center_p^*(G)$ to be the normal subgroup of $G$ such that $\Center_p^*(G)/\OO_{p'}(G)=\Center(G/\OO_{p'}(G))$.

We first state the original Glauberman's $Z^*$-Theorem whose proof does not depend on the classification of finite simple groups.
\begin{lem}\emph{(Glauberman's $Z^*$-Theorem)}\label{lem2} Let $G$ be a finite group and let $P$ be a Sylow $2$-subgroup of $G$. If $x\in P$  and $x^G\cap P=\{x\}$, then $x\in \Center_2^*(G)$.
\end{lem}

\begin{proof}
This is a restatement of Theorem $3$ in \cite{Glauberman}.
\end{proof}
The odd version of the Glauberman's $Z^*$-theorem, which is called the Glauberman's $Z_p^*$-theorem says that if $x\in P$ is an element of order $p$ and $x^G\cap P=\{x\}$, then $x\in\Center_p^*(G)$. The proof of this theorem depends on the classification (for a proof, see \cite[Theorem 4.1]{GR2}). By Sylow's theorem, it is easy to see that if $x^G\cap P=\
\{x\}$ then $x$ does not commute with any $G$-conjugate $x^g\neq x$ of $x$. 
Finally, the conclusion of  the Glauberman's $Z_p^*$-theorem can be written as $G=\Centralizer_G(x)\OO_{p'}(G)$. 

For an arbitrary $p$-element $x\in P$ which is not of prime order satisfying $x^G\cap P=\{
x\}$, 
to use the Glauberman's $Z_p^*$-Theorem, we need the following lemma.
\begin{lem}\label{lem: power up}
Let $G$ be a finite group and let $P$ be a Sylow $p$-subgroup of $G$ for some prime $p$. Let  $x\in P$.  If $x^G\cap P=\{x\}$, then $y^G\cap P=\{y\}$ for every $y\in \langle x\rangle$.
\end{lem}

\begin{proof}
Suppose that $x^G\cap P=\{x\}$ and $y\in\langle x\rangle.$ Then $P\leq \Centralizer_G(x)\leq \Centralizer_G(y)$. By Lemma \ref{lem: controls of p-fusion}, we need to show that $\Centralizer_G(y)$ controls $G$-fusion in $P$.
Let $z,z^g\in P$ for some $g\in G.$ By Lemma \ref{lem: controls of p-fusion}, $\Centralizer_G(x)$ controls $p$-fusion in $G$ so $z^g=z^t$ for some $t\in\Centralizer_G(x)$. As $\Centralizer_G(x)\leq \Centralizer_G(y)$, we have $t\in \Centralizer_G(y)$ and the claim follows.
\end{proof}

%---------------------
We will need the following results.
\begin{lem}\label{lem1} Let $G$ be a finite group and let $\pi$ be a non-empty set of primes.
\begin{enumerate}[$(1)$]
\item If $\mu\subseteq \pi$ is a non-empty subset, then $d_\pi(G)\leq d_\mu(G)\leq 1$. 
 \item If $N\unlhd G$, then $d_\pi(G)\leq d_\pi(G/N)d_\pi(N)$.
 \item If $G$ is a non-abelian $p$-group for some prime $p$, then $d(G)<(p+1)/p^2$.
 \item If $G$ does not have any normal Sylow $p$-subgroup for some prime $p$, then $d(G)\leq 1/p.$

\end{enumerate}
\end{lem}

\begin{proof} Part (1) can be found in \cite[Proposition $5$]{MH} and  Part (2) is Lemma $2.3$ in \cite{FG}.  Finally, the last two parts can be found in Lemma 2 in \cite{GR}.\end{proof}

Finite groups $G$ with $d(G)\ge 1/2$ were classified by Lescot in \cite{Lescot} and \cite{Lescot2}. To state the result, we need the following notation. For any integer $m\ge 1$, denote by $G_m$ the group defined by \[G_m=\langle a,b:a^3=b^{2^{m}}=1,a^b=a^{-1}\rangle.\] Note that $G_1\cong \Sym_3.$ We have that $|G_m|=3\cdot 2^m, \Center(G_m)=\langle b^2\rangle ,G_m'=\langle a\rangle$, and  $G_m/\Center(G_m)\cong \Sym_3$.

\begin{lem}\label{lem: classification} Let $G$ be a finite group. Then $d(G)\ge 1/2$ if and only if one of the following holds
\begin{enumerate}[$(i)$]
\item $G$ is abelian and $d(G)=1$.
\item $G\cong P\times A$, where $A$ is abelian of odd order and $P$ is a Sylow $2$-subgroup of $G$ with $|G'|=|P'|=2$ and $d(G)=d(P)=(1+4^{-m})/2$ and $G/\Center(G)$ is elementary abelian of order $4^m$ for some integer $m\ge 1$. Moreover, $1/2<d(G)\leq 5/8.$
\item $G\cong G_m\times A$ and $d(G)=1/2$, where $A$ is abelian and $m\ge 1.$
\end{enumerate}
\end{lem}

\begin{proof}
This is a combination of  Theorem 3.1 in \cite{Lescot2} and Corollary 3.2 in \cite{Lescot}.
\end{proof}

It follows from Lemma \ref{lem: classification} that there is no $2$-groups $G$ with $d(G)=1/2$. Also, if $G$ is of odd order with $d(G)\ge 1/2$, then $G$ is abelian and $d(G)=1$.

%The following lemma will be very useful.
\begin{lem}\label{lem: quotients}
Let $G$ be a finite group and let $p$ be a prime. If $N\unlhd G$ is a $p'$-subgroup, then $k_p(G)=k_p(G/N)$ and so $d_p(G)=d_p(G/N)$.
\end{lem}

\begin{proof}
Let $N$ be a normal $p'$-subgroup of $G$ and let $P\in\Syl_p(G)$. Write $\overline{G}=G/N$ and use the `bar' notation. Since $p\nmid |N|$, the Sylow $p$-subgroups of $G$ and $\overline{G}$ have the same order, thus it suffices to show that $k_p(G)=k_p(\overline{G})$.
By Lemma \ref{lem1}(2), we have $d_p(G)\leq d_p(\overline{G})$ since $d_p(N)\leq 1.$ It follows that $k_p(G)\leq k_p(\overline{G})$. The other direction is obvious.
%However, if $\overline{C}$ is a $\overline{G}$-conjugacy class of $p$-elements in $\overline{G}$, then we can choose a representative $\overline{x}\in\overline{C}$ such that ${x}\in P$ is a $p$-element and so $\overline{C}$ contains the image in $\overline{G}$ of the $G$-class $x^G.$  Moreover, if  the two distinct $\overline{G}$-classes containing $\overline{x}$ and $\overline{y}$ with $x,y\in P$, then $x^G$ and $y^G$ are two distinct $G$-classes of $G$. Hence $ k_p(\overline{G})\leq k_p(G)$. Therefore $k_p(G)=k_p(\overline{G})$.
\end{proof}

%------------Theorem A, B and C------------------
\section{Conjugacy classes of $2$-elements}\label{sec3}
We will prove Theorems A, B and C in this section. Recall that a $p$-group $P$ is said to be extra-special if $P'=\Phi(P)=\Center(P)$ and $|\Center(P)|=p.$ The following lemma is key to our proofs.

\begin{lem}\label{lem4} Let $G$ be a finite  group and let $P$ be a Sylow $2$-subgroup of $G.$  Suppose that $\OO_{2'}(G)=\Center(G)=1$ and $|P|>1$. Then $k_2(G)\leq {|P|}/{2}$ and $d_2(G)\leq 1/2$. Moreover, if $k_2(G)=|P|/2$, then  one of the following holds.
\begin{enumerate}[$(1)$]
\item  $k_2(G)=2$ and $P$ is elementary abelian of order $4$; or 
\item $k_2(G)>2$ and $P$ is an extra-special group  of order $2^{1+2m}$  for some integer $m\ge 1$ with $\Center(P)=\langle z\rangle$ a cyclic group of order $2$. Moreover, $|z^G\cap G|=3$  and for any $1\neq y\in P$ with $ y\not\in z^G$, $|y^G\cap P|=2$.
\end{enumerate}
\end{lem}

\begin{proof} 
The hypothesis of the lemma yields that $\Center_2^*(G)=1$. Let $P$ be a Sylow $2$-subgroup of $G.$ By Lemma \ref{lem2}, if $1\neq x\in P$, then $|x^G\cap P|\ge 2.$
Let $k=k_2(G)-1$ be the number of nontrivial conjugacy classes of
$2$-elements in $G.$ Clearly, we can choose a  complete set $\Gamma=\{x_i\}_{i=1}^k$ of representatives for
all nontrivial conjugacy classes of $2$-elements in $G$ such that $\Gamma\subseteq P\setminus\{1\}.$
Notice that $k\ge 1$ as otherwise $P$ is trivial.

Observe that $|x_i^G\cap P|\geq 2$ for all $i$ with $1\leq i\leq k,$
 $P\setminus\{1\}=\cup_{i=1}^k x_i^G\cap P$ and $x_i^G\cap
x_j^G=\emptyset\text{ for all } 1\leq i\neq j\leq k$. We have that
$$|P\setminus\{1\}|=\sum_{i=1}^k |x_i^G\cap P|\geq \sum_{i=1}^k 2= 2k.$$ Hence
$|P|-1\geq 2k$ and so $2k\leq |P|-2$ since $|P|-1$ is odd. Thus
$k_2(G)=k+1\leq {|P|}/{2}$ and $d_2(G)\leq 1/2$ as wanted. 

Next, assume that $k_2(G)=|P|/2$.  We know that $|x_i^G\cap P|\ge 2$ for all $i=1,2,\ldots, k.$
If $k=1$, then $k_2(G)=2$ and $|P|=4$, so $|x_1^G\cap P|=3$ and $P$ is elementary abelian of order $4$. Thus part (1) holds. Assume that $k\ge 2.$ Since $|x_i^G\cap P|\ge 2$ for every $i$ and $k=|P|/2-1$, we obtain that $|x_j^G\cap P|=3$ for a unique index $j$ and $|x_i^G\cap P|=2$ for all $1\leq i\neq j\leq k.$ So we may assume that $|x_1^G\cap P|=3$ and $|x_i^G\cap P|=2$ for $2\leq i\leq k.$ 

By Corollary \ref{cor}, we have $1/2=d_2(G)\leq d(P)$ and thus either $P$ is abelian or $|P'|=2,d(P)=(1+4^{-m})/2$ and $P/\Center(P)$ is elementary abelian of order $4^m$ by Lemma \ref{lem: classification}. We claim that $P$ is non-abelian. By way of contradiction, assume that $P$ is abelian. Let $H=\Normalizer_G(P)$.  By Lemma \ref{lem: fusion}(1), $H$ controls $G$-fusion in $P$ (since $P$ is abelian) and so $x_k^G\cap P=x_k^H$.   Moreover, as $P$ is abelian, $P\leq \Centralizer_H(x_k)\leq H$ and hence $|x_k^H|\ge 1$ is odd. However $|x_k^H|=|x_k^G\cap P|=2$ by the result in the previous paragraph, which is a contradiction. 

Next, we claim that $P$ is extra-special. It suffices to show that $\Center(P)=P'$. Write $P'=\langle z\rangle$. Since $|P'|=2$, we have $P'\leq \Center(H)\cap\Center(P)$. Let $1\neq u\in \Center(P)$. We claim that $|u^G\cap P|=3$. Assume by contradiction that $u^G\cap P=\{u,v\}$, where $u\neq v\in P.$ If $v\in\Center(P)$, then $v=u^h$ for some $h\in H$ by Lemma \ref{lem: fusion}(1). It follows that $|u^H|=2$ which is impossible as $P\leq \Centralizer_H(u)\leq H$. Thus $v\not\in \Center(P)$. Now $u^G\cap P=v^G\cap P=\{u,v\}$. Since $v\in P\setminus \Center(P)$, we have $|v^P|>1$ whence $v^P=\{u,v\}$.  In particular $u=v^t$ for some $t\in P.$ Hence $v=u^{t^{-1}}=u$ as $u\in\Center(P)$. This contradiction shows that $|u^G\cap P|=3$ for every $1\neq u\in\Center(P)$. In particular,  $|z^G\cap P|=3.$ Now if $\Center(P)\neq P'$, then we can choose $u\in\Center(P)\setminus P'$ and by our previous claim, $|u^G\cap P|=3$. It follows that $u$ and $z$ are $G$-conjugate  as there is only one class of $2$-elements satisfying the previous condition. Again this is a contradiction by using Lemma \ref{lem: fusion}(1)  and the fact that $z\in\Center(H)$. The proof is now complete.
\end{proof}

%------Theorem A
We are now ready to prove our first theorem.
\begin{proof}[\textbf{Proof of Theorem A}] 
Let $G$ be a finite group. Assume first that $G$ has a normal $2$-complement and $d(P)>1/2$ for some Sylow $2$-subgroup $P$ of $G$. By Corollary \ref{cor}(1), $d_2(G)=d(P)>1/2$.
Conversely, assume that $d_2(G)>1/2$. Let $P$ be a Sylow $2$-subgroup of $G$. If $G$ has a normal $2$-complement, then $d(P)=d_2(G)>1/2$. Thus we only need to show that $G$ has a normal $2$-complement. We proceed by induction on $|G|$. Observe that if $N$ is a proper nontrivial normal subgroup of $G$, then $d_2(N)$ and $d_2(G/N)$ are strictly larger than $1/2$ by Lemma \ref{lem1}(2). By induction, both $N$ and $G/N$ have normal $2$-complements. Hence if $N$ is of odd order or $G/N$ is a $2$-group then $G$ has a normal $2$-complement and we are done. Therefore, we may assume that $\OO_{2'}(G)=1$ and $G=\OO^{2}(G)$.

 Suppose  that $\Center(G)$ is nontrivial.   As $\OO_{2'}(G)=1$, $\Center(G)$ must be a $2$-group. Now $G/\Center(G)$ has a normal $2$-complement, say $K/\Center(G)\unlhd G/\Center(G)$, for some normal subgroup $K$ of $G$ with $\Center(G)\leq K$. Hence $\Center(G)\unlhd K\unlhd G$ and $G/K$ is a $2$-group. Since $G=\OO^{2}(G)$, we obtain that $G=K$. We now see that $\Center(G)$ is a normal Sylow $2$-subgroup of $G$ and thus $G$ has a normal $2$-complement by Lemma \ref{lem: fusion}(2). So we may assume that $\Center(G)=1.$
Now Lemma \ref{lem4} yields a contradiction. 
\end{proof}

%--------Theorem B

We now study the structure of finite groups $G$ with $d_2(G)=1/2.$ We first consider the solvable case.

\begin{lem}\label{lem:2-class-solvable}
Let $G$ be a finite solvable group. Suppose that $\OO_{2'}(G)=1$ and $G=\OO^2(G)$. Then $d_2(G)=1/2$ if and only if $G\cong \Alt_4$.
\end{lem}

\begin{proof}
If $G\cong \Alt_4$, then $d_2(G)=1/2$ as $k_2(G)=2$ and $|P|=4$. Conversely, assume that $d_2(G)=1/2$. We proceed by induction on $|G|$. By Corollary \ref{cor}, we have $1/2=d_2(G)\leq d(P)$. Lemma \ref{lem: classification} yields that either $P$ is abelian or $|P'|=2$ and $P/\Center(P)$ is elementary abelian of order $4^m.$ In both cases, $d(P)>1/2.$ It follows that $G$ is not a $2$-group and so $P$ is non-cyclic by Corollary 5.14 in \cite{Isaacs}. As $\OO_{2'}(G)=1$, $\Centralizer_G(\OO_2(G))\subseteq \OO_2(G)$ by \cite[Theorem 3.21]{Isaacs}, hence $\Center(G)\leq P.$ 

We claim that $G/Z$ satisfies the hypothesis of the lemma for any central subgroup $Z\leq \Center(G)$. Clearly, $\OO^2(G/Z)=G/Z$ since $G=\OO^2(G)$. Next, assume that  $K/Z=\OO_{2'}(G/Z)$, where $Z\leq K\unlhd G$. Then $K$ has a central Sylow $2$-subgroup $Z$ and so by Lemma \ref{lem: fusion}(2), $K$ has a normal $2$-complement $\OO_{2'}(K)$. Since $K\unlhd G,$ $\OO_{2'}(K)\leq \OO_{2'}(G)=1$. Hence $K=Z$ and so $\OO_{2'}(G/Z)=1$. 

By Lemma \ref{lem1}(2), we have $1/2=d_2(G)\leq d_2(G/Z)d_2(Z)=d_2(G/Z)$. If $d_2(G/Z)>1/2$, then $G/Z$ has a normal $2$-complement by Theorem A but this would imply that $G/Z$ is a $2$-group and so $G$ is a $2$-group, a contradiction. Hence $d_2(G/Z)=1/2.$ Therefore, by using induction on $|G|$, if $Z$ is nontrivial, then $G/Z\cong \Alt_4.$ We now consider two cases separately, according to whether  $P$ is abelian or not. 

\textbf{Case $1$}: $P $ is abelian. As $\Centralizer_G(\OO_2(G))\subseteq \OO_2(G)$, we have $P=\OO_2(G)$ and so $P=\Centralizer_G(P)\unlhd G$. Clearly,  $P\neq \Center(G)$, as otherwise $G$ is a $2$-group by applying Lemma \ref{lem: fusion}(2). Thus $|P:\Center(G)|\ge 2$.

%For every $x\in P\setminus Z$, we have $P\leq \Centralizer_G(x)$ and $x^G\subseteq P$ which implies that $|x^G|\ge 3$ (as $|x^G|>1$ is odd)
%and thus $|P|\geq |Z|+3k$, where $k$ is the number of $G$-conjugacy classes of non-central $2$-elements of $G$. It follows that $k_2(G)=|Z|+k\leq (|P|+2|Z|)/3$ which implies that $|P:Z|\leq 4.$
%Thus either $|P:Z|=2$ or $4$. Clearly, the former case cannot occur as otherwise $G/Z$ would have a normal $2$-complement. Therefore $|P:Z|=4.$ 

Assume first that $\Center(G)=1$. By Lemma \ref{lem4}, $P$ is elementary abelian of order $4$ and $k_2(G)=2$. As $P=\Centralizer_G(P)\unlhd G,$ $G/P$ embeds into $\GL_2(2)\cong\Sym_3$. Since $G/P$ is of odd order and nontrivial, $G/P\cong C_3$. It is not hard to see that $G\cong\Alt_4$.

Next, assume that $\Center(G)$ is nontrivial.  Then $G/\Center(G)\cong\Alt_4$. 
By  \cite[Theorem 5.18]{Isaacs}, $G'\cap \Center(G)=G'\cap P\cap \Center(G)=1$. Let $R$ be a Sylow $3$-subgroup of $G$. Then $G=PR,|R|=3$ and $R$ acts nontrivially and coprimely on $P$; hence $\Center(G)= \Centralizer_P(R)$ and $[P,R]= G'\leq P$. Since $R$ acts coprimely on $P,$ we have $P=[P,R]\times \Centralizer_P(R)=G'\times \Center(G).$ Moreover, $G'R\unlhd G$ and $|G/G'R|$ is a $2$-power, so $G=G'R$ forcing $\Center(G)=1$, a contradiction.

\medskip
\textbf{Case $2$}: $P $ is non-abelian. We have $P'\leq \Center(P)\leq \Centralizer_G(\OO_2(G))\leq \OO_2(G)$. Observe that $G/\OO_{2,2'}(G)$ has an abelian Sylow $2$-subgroup, so $G/\OO_{2,2'}(G)$ has a normal Sylow $2$-subgroup by using Hall-Higman Lemma 1.2.3 (\cite[Theorem 3.21]{Isaacs}); hence $G=\OO_{2,2',2,2'}(G)$. (For the definitions of $\OO_{2,2'}(G)$ and  $\OO_{2,2',2,2'}(G)$,  see \cite[6.3]{Gorenstein}.)

Assume first that $G=\OO_{2,2'}(G)$. Then $\Centralizer_G(P)\leq P\unlhd G.$ It follows that $P'\leq \Center(G)$ as $|P'|=2$.  Thus $G/P'\cong \Alt_4$ and $|G|=24$. It is easy to check that $G\cong 2\cdot \Alt_4\cong\SL_2(3)$ as $G=\OO^2(G)$. However, $d_2(\SL_2(3))=3/8<1/2.$

Assume that  $G/\OO_{2,2'}(G)$ is nontrivial.  Let $L=\OO_{2,2'}(G)\unlhd G.$  Since $\OO_{2'}(L)=1$, by using Theorem A and Lemma \ref{lem1}(2), we see that $d_2(L)= 1/2$.  But then this forces $d_2(G/L)=1$. By Corollary \ref{cor}(2), $G/L$ has a normal $2$-complement and since $G=\OO^2(G)$, we deduce that $G/L$ is a $2'$-group, forcing $G=L$, which is a contradiction.
This completes our proof.
\end{proof}

\begin{lem}\label{lem: solvable groups}
Let $G$ be a finite solvable group. Suppose that $\OO_{2'}(G)=1$. If $d_2(G)=1/2$, then $G/\Center(G)\cong\Alt_4$ or $\Sym_4.$
\end{lem}

\begin{proof}  Suppose that  $G$ is a finite solvable group with $d_2(G)=1/2$ and $\OO_{2'}(G)=1.$ Let $L=\OO^2(G)$. Then $\OO_{2'}(L)=1$ and $L=\OO^{2}(L)$. By  Lemma \ref{lem1}(2), $d_2(L)\ge 1/2.$ If $d_2(L)>1/2$, then $L$ has a normal $2$-complement by Theorem A. However, as $\OO_{2'}(L)=1$, $L$ must be a $2$-group and hence $G$ is a $2$-group with $d_2(G)=d(G)=1/2$, which is impossible by Lemma \ref{lem4}. Therefore, $d_2(L)=1/2.$ Hence $L\cong \Alt_4$ by Lemma \ref{lem:2-class-solvable}. 

Let $C=\Centralizer_G(L)\unlhd G$. As $\Center(L)=1$, we have $C\cap L=1$. Then $\Alt_4\cong LC/C\unlhd G/C\leq \Aut(\Alt_4)=\Sym_4$. Hence $G/C \cong \Alt_4$ or $\Sym_4$. It remains to show that $C=\Center(G)$. Let $P$ be a Sylow $2$-subgroup of $G$. 

As $C\times L=CL\unlhd G$, we have $1/2=d_2(G)\leq d_2(CL)\leq d_2(L)d_2(C)=d_2(C)/2$ and so $d_2(C)=1$. Thus $C$ has a normal $2$-complement and an abelian Sylow $2$-subgroup. However, as $\OO_{2'}(C)\leq\OO_{2'}(G)=1$, $C$ must be an abelian $2$-group and $C\leq P$. We also have that $1/2=d_2(G)\leq d_2(G/L)d_2(L)=d_2(G/L)/2$, so $d_2(G/L)=1$ where $G/L$ is a $2$-group. It follows that $G/L$ is an abelian $2$-group. In particular, $G'\leq L$. Thus $[P,C]\subseteq G'\cap C\subseteq L\cap C=1$, so $[P,C]=1$. As $[L,C]=1$ and $G=PL$, we have $C\leq \Center(G)$. Since $\Center(G/C)$ is trivial, we must have $C=\Center(G)$ as wanted.
\end{proof}

We next classify all finite non-abelian simple groups $S$ such that $d_2(S)=1/2$.

\begin{lem}\label{lem: simple groups}
Let $S$ be a finite non-abelian simple group. Then $d_2(S)=1/2$ if and only if $S\cong \PSL_2(q)$ with $3<q\equiv 3,5$ $({mod}\: 8)$, where $q$ is a prime power.
\end{lem}

\begin{proof} Let $S$ be a finite non-abelian simple with a Sylow $2$-subgroup $P.$ If  $S\cong \PSL_2(q)$ with $3<q\equiv 3,5$ ({mod}\: $8$), then $P$ is elementary abelian of order $4$ and $S$ has only one class of involutions so $k_2(S)=2$ and thus $d_2(S)=1/2.$ Conversely, assume that $S$ is a finite non-abelian simple group with $d_2(S)=1/2$. By Lemma \ref{lem4}, either $P$ is elementary abelian of order $4$ with $k_2(P)=2$ or $P$ is extra-special of order $2^{1+4m}$ for some integer $m\ge 1.$

Assume first that $P$ is elementary abelian of order $4$. It follows from \cite[Theorem I]{Walter} that $S$ is isomorphic to $\PSL_2(2^f), (f\ge 1)$, $\PSL_2(q)$ with $3<q\equiv 3,5 $ (mod $8$) ${}^2\textrm{G}_2(3^{2n+1}), n\ge 1$ or $\textrm{J}_1$.  Since $|P|=4$, we deduce that $S\cong\PSL_2(q)$ with $3<q\equiv 3,5 $ (mod $8$). Notice that $\PSL_2(4)\cong\PSL_2(5)$. 

Assume now that $P$ is extra-special of order $2^{1+4m},m\ge 1.$ In this case, $P$ is nilpotent of class $2$.  It follows from \cite[Main Theorem]{GG} that $S$ is isomorphic to one of the groups $\PSL_2(q)$ with with $q\equiv 7,9$ $(\text{mod}\: 16)$, $\Alt_7$, $\textrm{Sz}(2^n)$, $\PSU_3(2^n)$, $\PSL_3(2^n)$ or $\textrm{PSp}_4(2^n)$ with $n\ge 2.$ However, except for the first two groups, the centers of the Sylow $2$-subgroups of the remaining simple groups have order at least $4$. For $\Alt_7$, we can check that $k_2(\Alt_7)=3$  so $d_2(\Alt_7)=3/8$ as a Sylow $2$-subgroup of $\Alt_7$ is isomorphic to $\textrm{D}_8$. Similarly, the Sylow $2$-subgroup of $S=\PSL_2(q)$ with $q\equiv 7,9$ (mod $8$) is also isomorphic to $\textrm{D}_8$. Again, except for the identity, $S$ has two non-trivial classes of $2$-elements, one consisting of all involutions in $S$ and another consisting of elements of order $4$. Thus these cases cannot occur.
\end{proof}
For a finite group $G$, we denote by $\Sol(G)$ the solvable radical of $G$, that is, the largest solvable normal subgroup of $G$. 
\begin{lem}\label{lem:almost simple}
Let $G$ be a finite  group.  Suppose that $\Sol(G)=1$ and $d_2(G)=1/2$. Then $G$ is a finite almost simple group. 
\end{lem}

\begin{proof}
Let $M$ be a minimal normal subgroup of $G$. As $G$ has a trivial solvable radical, $M\cong S^k$, where $S$ is a non-abelian simple group and $k\ge 1$ is an integer. By Lemma \ref{lem1}(2), we have $1/2=d_2(G)\leq d_2(G/M)d_2(M)\leq d_2(M)$. By applying this lemma repeatedly, we have $1/2\leq d_2(M)\leq d_2(S)^k$. By Lemma \ref{lem4}, $d_2(S)\leq 1/2$; so $1/2\leq d_2(S)^k\leq (1/2)^k$, forcing $k=1$ and $d_2(S)=1/2$.

Let $C=\Centralizer_G(M)$. Then $C\unlhd G$ and $CM=C\times M\unlhd G.$ By Lemma \ref{lem1}(2),  $$1/2=d_2(G)\leq d_2(G/MC)d_2(MC)\leq d_2(MC)\leq d_2(M)d_2(C)=d_2(C)/2.$$ Hence $d_2(C)=1$  and so $C$ is  solvable Corollary \ref{cor}(2) and Feit-Thompson theorem. Since $\Sol(G)=1$ and $C$ is a solvable normal subgroup of $G$, we must have $C=1$ so $G$ is almost simple with simple socle $M$.
\end{proof}

\begin{lem}\label{lem:reduction}
Let $G$ be a finite perfect  group.  Suppose that $\OO_{2'}(G)=1$ and $d_2(G)=1/2$. Then
$G\cong \PSL_2(q)$ with $3<q\equiv 3,5$ $({mod}\: 8)$, where $q$ is a prime power.

\end{lem}

\begin{proof} 
Let $U$ be the solvable radical of $G.$ Then $G/U$ is non-solvable. Since $1/2=d_2(G)\leq d_2(U)d_2(G/U)$ by Lemma \ref{lem1}, both $d_2(U)$ and $d_2(G/U)$ are at least $1/2$. By Theorem A,  $d_2(G/U)=1/2$ as otherwise $G/U$ is solvable.  By Lemmas \ref{lem: simple groups} and \ref{lem:almost simple} and the fact that $G$ is perfect, $G/U\cong\PSL_2(q)$, where $q\equiv 3,5$ (mod $8$). We have $d_2(U)=1$ and since $\OO_{2'}(G)=1$, $U$ is an abelian $2$-group. We will show that $G$ is non-abelian simple by induction on $|G|$.

 If $U=1$, then $G$ is simple and we are done. Assume that $U$ is nontrivial. Assume first that $Z:=\Center(G)$ is nontrivial. Then $Z$ must be a $2$-group. Consider the quotient group $G/Z.$ Observe that $G/Z$ is perfect, $d_2(G/Z)=1/2$ and $\OO_{2'}(G/Z)=1$. Since $|G/Z|<|G|$, by induction hypothesis  $G/Z$ is non-abelian simple and thus $Z=U$. It follows that $G\cong\SL_2(q)$, the only Schur cover of $\PSL_2(q)$ with $q\equiv 3,5$ (mod $8$).  However, it is easy to see that $\SL_2(q)$ has only two classes of nontrivial $2$-elements and the Sylow $2$-subgroup of $\SL_2(q)$  with $q\equiv 3,5$ (mod $8$) has order $8$, so $d_2(G)=3/8<1/2,$ which is a contradiction. Hence we may assume that $\Center(G)=1$. Since $U$ is a normal abelian subgroup of $G$, we have $U\leq \Centralizer_G(U)\unlhd G$. Since $U$ is not central in $G$ and $G/U$ is non-abelian simple,  we must have that $U=\Centralizer_G(U)$.

Let $P\in{\Syl}_2(G)$. Note that the hypothesis of Lemma \ref{lem4} holds for $G$, that is, $\OO_{2'}(G)=\Center(G)=1$, $|P|>1$ and that $d_2(G)=1/2$. We claim that $k_2(G)>2$. Assume by contradiction that $k_2(G)=2$. Then $P$ is elementary abelian of order $4$ by Lemma \ref{lem4}. However the Sylow $2$-subgroup of $\PSL_2(q)$ with $q\equiv 3,5$ (mod $8$) has order $4$. So $U=1$, which is a contradiction. Therefore $k_2(G)>2$ and so part (2) of Lemma \ref{lem4} holds. Obviously $|U|\ge 4$ and $P'=\langle z\rangle=\Center(P) < U.$ If $z^G=U\setminus\{1\}$, then $U$ is elementary abelian of order $4$ and thus $G/U$ embeds into $\GL_2(2)$ which is impossible. Thus there exists $1\neq y\in U\setminus z^G$ and so $|y^G\cap P|=2$. Since $y\in U\unlhd G,$ we have $y^G\subseteq U\leq P$, so $|y^G\cap P|=|y^G|=2$ which implies that $U\leq \Centralizer_G(y)<G$ and $|G:\Centralizer_G(y)|=2$. Therefore, $G/U$ has a subgroup of index $2$ which is impossible as $G/U$ is non-abelian simple.
\end{proof}

%---------Theorem B-----------------------------
\begin{proof}[\textbf{Proof of Theorem B}]
Let $G$ be a finite group and assume that $d_2(G)=1/2$ and $\OO_{2'}(G)=1$.
If $G$ is solvable, then $G/\Center(G)\cong\Alt_4$ or $\Sym_4$ by Lemma \ref{lem: solvable groups}. So part (1) of the theorem holds. Assume that $G$ is non-solvable. Let $L$ be the last term of the derived series of $G$. By Theorem A and Lemma \ref{lem1}(2), $d_2(L)=1/2$. Moreover $L$ is perfect and $\OO_{2'}(L)=1$. By Lemma \ref{lem:reduction}, $L\cong S$ where $S=\PSL_2(q)$ $q\equiv 3,5$ (mod $8$). Write $q=p^f$, where $p$ is a prime and $f\ge 1$ is an integer. We see that $f$ must be odd and thus $\Out(S)=C_2\times C_f$.

Let $C=\Centralizer_G(L)$. Then $C\unlhd G$, $C\cap L=1$ and $G/C$ is an almost simple group with socle isomorphic to $S$. Since $d_2(L)=1/2$, we see that $d_2(C)=1$ and since $\OO_{2'}(G)=1$, $C$ is a normal abelian $2$-subgroup of $G$. We also have that $1/2=d_2(G)\leq d_2(G/L)d_2(L)=d_2(G/L)/2$ so $d_2(G/L)=1$ and so $G/L$ has a normal $2$-complement $W/L$ and an abelian Sylow $2$-subgroup $PL/L$ by Corollary \ref{cor}, where $P$ is any Sylow $2$-subgroup of $G$ containing $C$. Since $CL/L$ and $W/L$ are normal subgroups of $G/L$ and have coprime orders, we deduce that $[C,W]\leq L$. As $C\unlhd G,$ we have $[C,W]\leq L\cap C=1$. Thus $[C,W]=1$. On the other hand, $PL/L$ is abelian, thus $[C,P]\leq L$. With the same reasoning, we have $[C,P]\leq L\cap C=1$. Since $G=PW,$ we obtain that $[C,G]=1$. In particular, $C\leq \Center(G)$ and since $G/C$ is almost simple, we must have that $C=\Center(G)$. Therefore, we have shown that $G/\Center(G)$ is almost simple with socle $S$ as required.
\end{proof}

We will need the following result for our proof of Theorem C.

\begin{lem}\label{lem:odd order} Let $G$ be a finite group of odd order and let $\sigma$ be a non-empty set of primes. If $d_\sigma(G)\ge 1/2$, then $G$ has a normal $\sigma$-complement and an abelian Hall $\sigma$-subgroup.
\end{lem}

\begin{proof} By Feit-Thompson theorem, we know that $G$ is solvable. By Lemma \ref{lem1}(2),  if $N\unlhd G$, then $d_\sigma(N)\ge 1/2$ and $d_\sigma(G/N)\ge 1/2$. Assume that $G$ has a  normal $\sigma$-complement $K$. Let $H$ be a Hall $\sigma$-subgroup of $G$. We claim that $H$ is abelian. As $G/K\cong H$,  we have $1/2\leq d_{\sigma}(H)=d(H)$, where the last equality holds as $H$ is a $\sigma$-group.
Thus $d(H)\ge 1/2$ where $H$ is a group of odd order. By Lemma \ref{lem: classification}, $H$ must be abelian as wanted. Therefore, it suffices to show that $G$ has a normal $\sigma$-complement.  We will prove this claim by induction on $|G|$.

Let $N$ be a minimal normal subgroup of $G$.  Then $N$ is an elementary abelian $p$-subgroup for some odd prime $p$.  As $d_{\sigma}(G/N)\ge 1/2$, by induction on $|G|$, $G/N$ has a normal $\sigma$-complement, say $M/N$. If $p\not\in \sigma$, then $M$ is also a normal $\sigma$-complement of $G$, and we are done. Thus we may assume that $\OO_{\sigma'}(G)=1$ and  $p\in\sigma.$
We have $M\unlhd G$ and $d_\sigma(M)\ge 1/2$. Therefore, by induction again, $M$ has a normal $\sigma$-complement whenever $M<G;$ but then this would imply that $M$ is a $\sigma$-subgroup since $\OO_{\sigma'}(M)\subseteq \OO_{\sigma'}(G)=1$ and hence $G$ is a $\sigma$-group.
So, we can assume $M=G$, hence $G/N$ is a $\sigma'$-group.
 
Since $G/N$ is solvable, let $T/N$ be a maximal normal subgroup of $G/N$ of prime index $r\not\in \sigma$. Since $T\unlhd G$, we have $d_\sigma(T)\ge1/2$ and again by induction,  $T$ has a normal $\sigma$-complement which implies that $T=N$. Thus  $N$ is a maximal normal subgroup of $G$ and $|G/N|=r$ is a prime different from $p$.

If $\Centralizer_G(x)=G$ for some $1\neq x\in N$, then $\langle x\rangle=N\leq \Center(G)$ and $G\cong C_p\times C_r$ by Lemma \ref{lem: fusion}(2), which is a contradiction as $\OO_{\sigma'}(G)=1$. So, we may assume that $\Centralizer_G(x)<G$ for all $1\neq x\in N$. Since $N$ is maximal in $G$,  $\Centralizer_G(x)=N$ for every $1\neq x\in N$. Thus $G$ is a Frobenius group with Frobenius kernel $N$ and Frobenius complement isomorphic to $C_r$. Set $|N|=p^k$ for some integer $k\ge 1.$ We can see that $\sigma=\{p\}$ and that $k_p(G)=(p^k-1)/r+1$ and so $d_p(G)={1}/{r}+{(r-1)}/{{(rp^k)}}\ge {1}/{2}.$
Notice that $r\neq p\ge 3$ and $r\mid p^k-1$. We consider the following cases:

(1) $r=3$ and $p\ge 5$. In this case, we have $$d_p(G)\leq {1}/{3}+{2}/{3}\cdot {1}/{5}={7}/{15}<{1}/{2}.$$

(2) $r\ge 5$ and $p=3$. Since $r>p$, $k\ge 2.$ We have $$d_p(G)\leq {1}/{5}+ {1}/{9}={14}/{45}<{1}/{2}.$$

(3) $r\ge 5$ and $p\ge 5.$ Clearly, we have $$d_p(G)< {1}/{5}+ {1}/{5}={2}/{5}<{1}/{2}.$$

Thus we have shown that $G$ has a normal $\sigma$-complement as wanted.
\end{proof}

%-----------------Theorem C----------------
We are now ready to prove Theorem C.

\begin{proof}[\textbf{Proof of Theorem C}]
Let $G$ be a finite group and let $\pi$ be a set of primes containing $2$ and let $\sigma=\pi\setminus \{2\}$. Suppose that $d_\pi(G)>1/2.$
By Lemma \ref{lem1}(1), we have $1/2<d_\pi(G)\leq d_2(G)$ and thus by Theorem A, $G$ has a normal $2$-complement $K$ and by Lemma \ref{lem1}(2), we have $$1/2<d_\pi(G)\leq d_\pi(K)d_\pi(G/K)\leq d_\pi(K)=d_\sigma(K).$$ By Lemma \ref{lem:odd order} , $K$ has a normal $\sigma$-complement, say $N$ and an abelian Hall $\sigma$-subgroup $T$. It follows that $G=PTN$, where $N$ is also a normal $\pi$-complement of $G$.
\end{proof}

%-----------Odd Primes---------------------------------
\section{Conjugacy classes of $p$-elements with $p$ odd}\label{sec4}
We now consider odd primes. We start with the following easy result.
\begin{lem}\label{lem:p-groups} Let $p$ be an odd prime. Let $G$ be a finite group and let $P$ be a Sylow $p$-subgroup of $G$. If $d_p(G)\ge (p+1)/(2p)$, then $P$ is abelian. 
\end{lem}

\begin{proof}
Let $G$ be a finite group  such that $d_p(G)\ge (p+1)/(2p).$ By Corollary \ref{cor}, we have $d_p(G)\leq d(P)$ which implies that $d(P)\ge (p+1)/(2p)$. If $P$ is abelian, then we are done. So, assume that $P$ is non-abelian. By Lemma \ref{lem1}(3), we have $d(P)<(p+1)/p^2.$ Since $p$ is odd, we can check that $(p+1)/(2p)>(p+1)/p^2$ and so $d(P)<(p+1)/p^2<(p+1)/(2p)\leq d(P)$, which is a contradiction. 
\end{proof}

%------Theorem D
\begin{proof}[\textbf{Proof of Theorem D}]
%\begin{proof}[of Theorem~D]
Let $p$ be an odd prime. Let $G$ be a finite group. 
Assume that $G$ has a normal $p$-complement and an abelian Sylow $p$-subgroup $P$. By Corollary \ref{cor}(2), we have $d_p(G)=d(P)=1>(p+1)/(2p)$. Conversely, assume that $d_p(G)>(p+1)/(2p)$ and  let $P\in{\Syl}_p(G)$.
By Lemma \ref{lem:p-groups}, $P$ is abelian. It remains to show that $G$ has a normal $p$-complement. We proceed by using induction on $|G|$.

We first claim that $\OO_{p'}(G)=1$. Assume by contradiction that $\OO_{p'}(G)$ is nontrivial. By Lemma \ref{lem1}(2), $d_p(G)\leq d_p(G/\OO_{p'}(G))d_p(\OO_{p'}(G))\leq d_p(G/\OO_{p'}(G))$, so by induction $G/\OO_{p'}(G)$ has a normal $p$-complement; hence $G$ will have a normal $p$-complement. Thus we may assume that $\OO_{p'}(G)=1.$

We next claim that $G=\OO^p(G)$. Indeed, if $N=\OO^p(G)$ is a proper subgroup of $G$, then $(p+1)/(2p)<d_p(G)\leq d_p(N)$; thus by induction again, $N$ has a normal $p$-complement $\OO_{p'}(N)$. Clearly, this is also a normal $p$-complement of $G$. 

We now show that $G$ is $p$-solvable. In fact, suppose that $G$ is not $p$-solvable and let $M/N$ be a non-abelian chief factor of $G$ with $p$ dividing $|M/N|$. There exists a non-abelian simple group $S$ and an integer $k\ge 1$ such that $M/N\cong S^k.$ By applying Lemma \ref{lem1}(2) repeatedly, 
we have $(p+1)/2p<d_p(S)^k\leq d_p(S).$ (Note that $p$ divides $|S|$.)
Let $T\in{\Syl}_p(S)$ and let $H=\Normalizer_S(T)$. Clearly $T$ is abelian, so by Lemma \ref{lem: fusion}(1), $H$ controls $S$-fusion in $T.$ Thus $x^S\cap T=x^H\subseteq T$ for every $x\in T.$ Since $S$ is non-abelian simple, $\Center_p^*(S)=1$. Now Lemmas \ref{lem: controls of p-fusion} and \ref{lem: power up} together with  Glauberman $Z_p^*$-theorem implies that  $|x^S\cap T|\ge 2$ for all $1\neq x\in T$. It follows that $|T|-1\ge 2(k_p(S)-1)$. This implies that $k_p(S)\leq (|T|+1)/2$ and hence $$(p+1)/2p<d_p(S)\leq (|T|+1)/(2|T|)\leq (p+1)/(2p)$$ as $|T|\ge p.$ This contradiction shows that $G$ is $p$-solvable.

By Hall-Higman Lemma 1.2.3 (\cite[Lemma 3.21]{Isaacs}) and the fact that $P$ is abelian, we have $P\leq \Centralizer_G(\OO_p(G))\leq \OO_p(G)$, so $P=\OO_p(G)\unlhd G.$ Let $P/N$ be a chief factor of $G$. Assume that $N$ is nontrivial. Then $(p+1)/(2p)<d_p(G)\leq d_p(G/N)$ and so by induction $G/N$ has a normal $p$-complement $K/N.$ However, as $G=\OO^p(G)$, $G=K$ which is impossible. So, we can assume that $P$ is an elementary abelian minimal normal $p$-subgroup of $G$.

If $|x^G|=1$ for some $1\neq x\in P$, then $x\in \Center(G)\cap P$ which forces $P=\langle x\rangle\subseteq \Center(G)$. In this case $G$ has a normal $p$-complement by Lemma \ref{lem: fusion}(2). Hence we can also assume that $|x^G|\ge 2$ for all $1\neq x\in P$ whence $k_p(G)\leq (|P|+1)/2$. Since $d_p(G)>(p+1)/(2p)$,  $$|P|(p+1)/(2p)<(|P|+1)/2.$$ However, this inequality cannot occur as $|P|\ge p.$
\end{proof}

\begin{proof}[\textbf{Proof of Theorem E}] 
Let $\pi$ be a non-empty set of odd primes and let $p$ be the smallest member in $\pi$. Let $G$ be a finite group with $d_\pi(G)>(p+1)/2p.$ For every $ r\in\pi$, we see that $$(r+1)/(2r)\leq(p+1)/(2p)<d_\pi(G)\leq d_r(G)$$ by Lemma \ref{lem1}(1), so $d_r(G)>(r+1)/(2r)$. By Theorem D, $G$ has a normal $r$-complement and an abelian Sylow $r$-subgroup. It follows that $G$ has a normal $\pi$-complement $N=\OO_{\pi'}(G)\unlhd G$ and $G$ is $\pi$-solvable. By \cite[Theorem 3.20]{Isaacs}, $G$ has a Hall $\pi$-subgroup $H$. Clearly $G=HN$ and $G/N\cong H.$ Since $$(p+1)/(2p)<d_\pi(G)\leq d_\pi(H)d_\pi(N)\leq d_\pi(H)=d(H),$$ we deduce that $$d(H)>(p+1)/(2p)>1/p\ge 1/r$$ and so by Lemma \ref{lem1}(4), $H$ has a normal Sylow $r$-subgroup. It follows that $H$ is nilpotent and thus $H$ is abelian.
\end{proof}

\begin{proof}[\textbf{Proof of Theorem F}] 
 Let $G$ be a finite group with a Sylow $p$-subgroup $P$, where $p$ is an odd prime. Suppose that $d_p(G)=(p+1)/(2p)$ and $\OO_{p'}(G)=1$. By Lemma \ref{lem:p-groups}, $P$ is abelian and thus by Lemma \ref{lem: fusion}(1), $d_p(\Normalizer_G(P))=d_p(G)=(p+1)/(2p).$ By Theorem 7.4.4 in \cite{Gorenstein}, we have $P=P\cap N'\times P\cap \Center(N)$, where $N=\Normalizer_G(P).$ Set $Z=P\cap\Center(N)$ and $U=P\cap N'.$ We have that $k_p(N)\leq (|P|+|Z|)/2$ as for any $x\in P\setminus Z,$ $|x^N|\ge 2.$ It follows that $$(p+1)/(2p)\leq (|P|+|Z|)/(2|P|)$$ and thus $|P:Z|\leq p$. Clearly, $|P:Z|>1$ as otherwise $P\subseteq \Center(N)$ and thus $G$ has a normal $p$-complement which forces $G=P$ since $\OO_{p'}(G)=1$. But then $d_p(G)=1$, a contradiction. Thus $|P:Z|=p$; hence $|U|=p$ and $|Z|=|P|/p.$ Moreover $|x^N|=2$ for every $x\in P\setminus Z.$ Set $U=\langle y\rangle$. Then $\Centralizer_G(P)=\Centralizer_N(y)\unlhd N$. As $|y^N|=2$, we have $|N:\Centralizer_G(P)|=2$. Since $U=P\cap N'$ has order $p$, we see that $U=[P,N]$. Furthermore, by Theorem 7.4.4 in \cite{Gorenstein}, $P\cap G'=P\cap N'=U$ is cyclic of order $p$ and by Theorem 5.18 in \cite{Isaacs}, $G'\cap Z=1.$ 

Let $\Fit^*(G)$ be the generalized Fitting subgroup of $G$. Then $\Fit^*(G)=\Fit(G)\Comp(G)$ is the central product of the Fitting subgroup $\Fit(G)$ and the layer $\Comp(G)$ of $G$, which is the product of all components of $G$, that is,  subnormal quasi-simple subgroups of $G$. Bender's theorem (\cite[Theorem 9.8]{Isaacs}) says that $\Centralizer_G(\Fit^*(G))\subseteq \Fit^*(G)$.

Assume first that $\Comp(G)=1$. Then $\Fit^*(G)=\Fit(G)$ is a $p$-group since $\OO_{p'}(G)=1$. As $\Centralizer_G(\Fit(G))\subseteq \Fit(G)$ and $P$ is abelian, $P=\Fit(G)$  and thus  $P=\Centralizer_G(P)\unlhd G$. It follows that $|G:P|=2$ and $G'=[G,P]=U$ is cyclic of order $p$; moreover $Z=\Center(G)\cap P=\Center(G)$ as $P$ is self-centralizing. Now $G/G'$ is an abelian group of order $2|Z|$. Hence $G/G'$ has a normal Sylow $2$-subgroup $A/G'$ and a normal Sylow $p$-subgroup $P/G'=ZG'/G'$, so $G/G'=A/G'\times ZG'/G'$ which implies that $G=Z\times A$, where $A$ is a nonabelian group of order $2p$ and it has a normal cyclic Sylow $p$-subgroup of order $p$. It is easy to see that $A\cong \textrm{D}_{2p}$, the dihedral group of order $2p$.

 Assume now that $E:=\Comp(G)$ is nontrivial. Since $G'\cap P$ is cyclic of order $p$, the center  of $E$ is either trivial or cyclic of order $p.$ If $|\Center(E)|=p$, then $E/\Center(E)$ is a $p'$-group which is impossible by Corollary 5.4 in \cite{Isaacs}. Thus $\Center(E)=1$. Hence $E$ has a Sylow $p$-subgroup of order $p$ which forces $E$ to be a non-abelian simple group. Now we have that $\Fit^*(G)=E\times F$ where $F=\Fit(G)\leq P$ is a $p$-subgroup. Since $E\cap P\leq G'\cap P=U$ is of order $p$, we have $U=E\cap P=G'\cap P$. Therefore, $F\cap G'\leq F\cap G'\cap P=E\cap P\cap F=1$, so $F\cap G'=1$ whence $F\leq \Center(G)$.
 Hence $\Centralizer_G(E)=F=\Center(G)$ and so $G/F$ is an almost simple group with socle isomorphic to $E$. Since $E$ has a cyclic Sylow $p$-subgroup of order $p$, we deduce from Lemma 2.3 in \cite{Tongviet} that $|\Out(E)|$ is prime to $p$. In particular, $G/EF$ is a solvable $p'$-group. Thus $G/E$ has a central Sylow $p$-subgroup $EF/E\cong F.$ By Lemma \ref{lem: fusion}(2), $G/E$ has a normal $p$-complement $A/E$ and $G/E=EF/E\times A/E$. Since $E\cap F=1$, we have $G=A\times F$ and so $A$ is almost simple with socle $E$.
\end{proof}

%%%%%%%%%%%%%%%%%%%%%%%%%%%%%%%%%%%%%%%%%%%%%%%%%%%%%%%%%%%%%%%%%%%%%%%%
%%%%%%%%%%%%%%%%%%%%%%%%%%%
%\section*{Acknowledgment}


\begin{thebibliography}{99}

\bibitem{FG} J. Fulman\ and\ R. Guralnick, Bounds on the number and sizes of conjugacy classes in finite Chevalley groups with applications to derangements, Trans. Amer. Math. Soc. {\bf 364} (2012), no.~6, 3023--3070. 

\bibitem{GG} R. Gilman\ and\ D. Gorenstein, Finite groups with Sylow $2$-subgroups of class two. I, II, Trans. Amer. Math. Soc. {\bf 207} (1975), 1--101; ibid. {\bf 207} (1975), 103--126.

\bibitem{Glauberman} G. Glauberman, Central elements in core-free groups, {\em J. Algebra} {\bf4} (1966) 403--420.

\bibitem{Gorenstein} D. Gorenstein, {\it Finite groups}, Harper \& Row, Publishers, New York, 1968.

\bibitem{GR2} R. M. Guralnick\ and\ G. R. Robinson, On extensions of the Baer-Suzuki theorem, Israel J. Math. {\bf 82} (1993), no.~1-3, 281--297.

\bibitem{GR} R. M. Guralnick\ and\ G. R. Robinson, On the commuting probability in finite groups, J. Algebra {\bf 300} (2006), no.~2, 509--528.

\bibitem{Gus} W. H. Gustafson, What is the probability that two group elements commute?, Amer. Math. Monthly {\bf 80} (1973), 1031--1034. 

\bibitem{Isaacs} I. M. Isaacs, {\it Finite group theory}, Graduate Studies in Mathematics, 92, American Mathematical Society, Providence, RI, 2008.


\bibitem{KNST} B. K\"{u}lshammer, G. Navarro, B. Sambale, P. H. Tiep, Finite groups with two conjugacy classes of $p$-elements and related questions for $p$-blocks, Bull. Lond. Math. Soc. {\bf 46} (2014), no.~2, 305--314. 

\bibitem{Lescot} P. Lescot, Isoclinism classes and commutativity degrees of finite groups, J. Algebra {\bf 177} (1995), no.~3, 847--869.
\bibitem{Lescot2} P. Lescot, Central extensions and commutativity degree, Comm. Algebra {\bf 29} (2001), no.~10, 4451--4460.

\bibitem{MH} A. Mar\'{o}ti\ and\ H. N. Nguyen, On the number of conjugacy classes of $\pi$-elements in finite groups, Arch. Math. (Basel) {\bf 102} (2014), no.~2, 101--108.

\bibitem{Tongviet} H. P. Tong-Viet, Brauer characters and normal Sylow $p$-subgroups, J. Algebra {\bf 503} (2018), 265--276. 
\bibitem{Walter} J. H. Walter, The characterization of finite groups with abelian Sylow $2$-subgroups, Ann. of Math. (2) {\bf 89} (1969), 405--514.
\end{thebibliography}
\end{document}